\documentclass[12pt]{amsart}
\usepackage{amsmath,amssymb,amscd}
\newtheorem{thm}{Theorem}[section]
\newtheorem{prop}[thm]{Proposition}

\newtheorem{corollary}[thm]{Corollary}
\newtheorem{problem}[thm]{Problem}
\newtheorem{remark}[thm]{Remark}
\newtheorem{lemma}[thm]{Lemma}
\newtheorem{conj}[thm]{Conjecture}

\newcommand{\D}{\displaystyle}
\newcommand{\ip}[2]{\ensuremath{\langle #1 , #2 \rangle}}
\newcommand{\tp}{\texttt{p}}
\newcommand{\p}{\partial}

\renewcommand{\span}{\operatorname{span}}

\newcommand{\eucl}{\operatorname{eucl}}
\newcommand{\pb}{\operatorname{par}}

\newcommand{\ccs}{\operatorname{\mathcal{C}^2_{\textmd{sub}}}}

\theoremstyle{definition}
\newtheorem{definition}{Definition}
\theoremstyle{remark}

\numberwithin{equation}{section}
\begin{document}
\begin{abstract}By employing a Carnot parabolic maximum principle, we show existence-uniqueness of viscosity solutions to a class of equations modeled on the parabolic infinite Laplace equation in Carnot groups. We show stability of solutions within the class and examine the limit as $t$ goes to infinity. 
\end{abstract}
\title{The Parabolic Infinite-Laplace Equation in Carnot groups}
\author{Thomas Bieske}
\thanks{The first author was partially supported by a University of South Florida Proposal Enhancement Grant.}
\author{Erin Martin}
\address{Department of Mathematics and Statistics\\
University of South Florida\\ 
4202 E. Fowler Ave. CMC342\\
Tampa, FL 33620, USA}
\email{tbieske@usf.edu}
\address{Department of Mathematics and Physics\\
Westminster College\\
501 Westminster Ave\\ 
Fulton, MO 65251, USA}
\email{Erin.Martin@westminster-mo.edu}
\date{January 15, 2015}
\maketitle
\section{Motivation}
In Carnot groups, the following theorem has been established. 
\begin{thm}\cite{B:HG, W:W, B:MP} \label{subelliptic}
    Let $\Omega$ be a bounded domain in a Carnot group and let 
    $v:\partial \Omega \to \mathbb{R}$ be a continuous 
    function.
    Then the Dirichlet problem 
    \begin{eqnarray*}
\left\{ \begin{array}{cc}
\Delta_{\infty}u  =  0  & \textmd{in\ \ \ } \Omega \\
u = v & \textmd{on\ \ \ } \partial \Omega
\end{array} \right.
\end{eqnarray*}
has a unique viscosity solution $u_\infty$. 
\end{thm}
Our goal is to prove a parabolic version of Theorem \ref{subelliptic} for a class of equations (defined in the next section), namely 
\begin{conj}
Let $\Omega$ be a bounded domain in a Carnot group and let $T>0$.  Let $\psi\in C(\overline{\Omega})$ and let $g \in C(\Omega \times [0,T))$ Then the Cauchy-Dirichlet problem
\begin{equation}\label{maincon}
\left\{ \begin{array}{cl}
u_t - \Delta^h_{\infty}u =0& \hspace{10pt} in \hspace{5pt} \Omega \times (0,T),\\ 
u(x,0) =\psi(x) & \hspace{10pt} on \hspace{5pt} \overline{\Omega} \\
u(x,t)=g(x,t) & \hspace{10pt} on \hspace{5pt} \partial \Omega \times (0,T)
\end{array}\right.
\end{equation}
has a unique viscosity solution $u$. 
\end{conj}

In Sections 2 and 3, we review key properties of Carnot groups and parabolic viscosity solutions. In Section 4, we prove uniqueness and Section 5 covers existence. 
\section{Calculus on Carnot Groups}
We begin by denoting an arbitrary Carnot group in $\mathbb{R}^N$ by $G$ and its
corresponding Lie Algebra by $g$.   Recall that $g$ is nilpotent and
stratified, resulting in the decomposition $$g= V_1 \oplus V_2 \oplus
\cdots \oplus V_l$$ for appropriate vector spaces that satisfy the Lie
bracket relation $[V_1,V_j]=V_{1+j}.$ The Lie Algebra $g$ is associated
with the group $G$ via the exponential map $\exp: g \to G.$ Since
this
map is a diffeomorphism, we can choose a basis for $g$ so that it is
the identity map.  Denote this basis by $$X_1,X_2,\ldots ,
X_{n_1},Y_1,Y_2,\ldots ,Y_{n_2},Z_1,Z_2, \ldots , Z_{n_3}$$ so that
\begin{eqnarray*} 
V_1 & = & \textmd{span}\{X_1,X_2,\ldots ,X_{n_1}\} \\
V_2 & = & \textmd{span}\{Y_1,Y_2,\ldots ,Y_{n_2}\} \\
V_3 \oplus V_4 \oplus \cdots \oplus V_l & = &  
\textmd{span}\{Z_1,Z_2,\ldots ,Z_{n_3}\}.
\end{eqnarray*} 
We endow $g$ with an inner product $\ip{\cdot}{\cdot}$
 and related norm $\|\cdot\|$ so that this basis is orthonormal. 
 Clearly, the Riemannian dimension of $g$ (and so $G$) is $N
= n_1+n_2+n_3$. 
However, we will also consider the homogeneous dimension of $G$, denoted $\mathcal{Q}$,
which is given by $$\mathcal{Q}=\sum_{i=1}^l i \cdot \dim V_i.$$

Before proceeding with the calculus, we recall the group and metric space properties. Since
the exponential map is the identity, the group law is the
Campbell-Hausdorff formula (see, for example, \cite{BO:LAG}). For our
purposes, this formula is given by 
\begin{equation}\label{CCH}
p \cdot q =
p+q+\frac{1}{2}[p,q]+R(p,q)
\end{equation} 
where $R(p,q)$ are terms of order $3$ or
higher. The identity element of $G$ will be denoted by $0$ and called the
origin. There is also a natural metric on $G$, which is
the Carnot-Carath\'{e}odory distance, defined for the points $p$ and $q$
as follows: 
 \begin{equation*}
d_C(p,q)= \inf_{\Gamma} \int_{0}^{1} \| \gamma '(t) \| dt 
\end{equation*}
where the set $ \Gamma $
is the set of all  curves $ \gamma $ such that $ \gamma (0) = p,
 \gamma (1) = q $ and $\gamma '(t) \in V_1$.  
 By Chow's theorem (see, for example,
\cite{BR:SRG}) any two points
 can be connected by such a curve, which means $ d_C(p,q) $ is an honest
metric.  Define a Carnot-Carath\'{e}odory ball of radius $r$ centered at a
point $p_0$ by $$B(p_0,r)=\{p\in G : d_C(p,p_0) < r\}.$$

In addition to the Carnot-Carath\'{e}odory metric, there is a smooth (off the origin) gauge.  This gauge is defined for a point $p=(\zeta_1, \zeta_2, \ldots, \zeta_l)$
with $\zeta_i \in V_i$ by 
\begin{equation}\label{gaugedef}
\mathcal{N}(p)=\bigg(\sum_{i=1}^l  \|\zeta_{i}\|^{\frac{\D 2l!}{\D i}}
\bigg )^\frac{\D 1}{\D 2l!}
\end{equation}
and it induces a metric $d_{\mathcal{N}}$ that is bi-Lipschitz equivalent to the 
Carnot-Carath\'{e}odory metric and is given by
$$d_{\mathcal{N}}(p,q) = \mathcal{N}(p^{-1}\cdot q).$$ We define a gauge ball 
of radius $r$ centered at a
point $p_0$ by $$B_{\mathcal{N}}(p_0,r)=\{p\in G : d_{\mathcal{N}}(p,p_0) < r\}.$$

In this environment, a smooth function $u: G \to \mathbb{R}$ has the horizontal derivative given by
$$\nabla_0 u=(X_1u,X_2u, \ldots, X_{n_1}u )$$ and the symmetrized horizontal second derivative matrix, 
denoted by $(D^2u)^\star$, with entries
\begin{eqnarray*} 
((D^2u)^\star)_{ij} =  \frac{1}{2} (X_iX_ju+X_jX_iu) 
\end{eqnarray*} 
for $i,j=1,2,\ldots , n_1.$  We also consider the semi-horizontal derivative given by 
$$\nabla_1 u=(X_1u,X_2u, \ldots, X_{n_1}u, Y_1u,Y_2u, \ldots, Y_{n_2}u).$$

Using the above derivatives, we define the $h$-homogeneous infinite Laplace operator for $h\geq 1$ by
\begin{equation*}
\Delta^h_{\infty}f  = \|\nabla_0 f\|^{h-3} \sum_{i,j=1}^{n_1} X_ifX_jfX_iX_jf =  \|\nabla_0 f\|^{h-3}\ip{(D^2f)^{\star}\nabla_0f}{\nabla_0f}.
\end{equation*}
Given $T>0$ and a function $u:G \times [0,T] \to \mathbb{R}$, we may define the analogous subparabolic infinite Laplace operator by 
$$u_t-\Delta^h_{\infty}u$$ and we consider the corresponding equation
\begin{equation}\label{eqmain}
u_t-\Delta^h_{\infty}u=0.
\end{equation}
We note that when $h\geq 3$, this operator is continuous. When $h=3$, we have the subparabolic infinite Laplace equation analogous to the infinite Laplace operator in \cite{B:MP}. The Euclidean analog for $h=1$ has been explored in  \cite{JK:JK} and the Euclidean analog for $1<h<3$ in \cite{PV}. 
 
We recall that for any open set $\mathcal{O} \subset G$, the
function $f$ is in the horizontal Sobolev space $W^{1,\tp}(\mathcal{O})$ if
$f$ and $ X_if$ are in $ L^\tp(\mathcal{O}) $ for $i=1,2,\ldots , n_1$. Replacing $ L^\tp(\mathcal{O})$ by $L_{loc}^\tp(\mathcal{O})$, the space $ W_{loc}^{1,\tp}(\mathcal{O}) $ is defined similarly.  The space $W_{0}^{1,\tp}(\mathcal{O})$ is the closure in $W^{1,\tp}(\mathcal{O})$ of smooth functions with
compact support.  In addition, we recall a function $u:G\to
\mathbb{R}$ is $\ccs$ if $\nabla_1u$ and $X_iX_ju$ are continuous for
all $i,j=1,2,\ldots n_1$. Note that $\ccs$ is not equivalent to (Euclidean) $C^2$. 
For spaces involving time, the space $C(t_1,t_2;X)$ consists of all continuous functions $u:[t_1,t_2]\to X$ with $\max_{t_1\leq t\leq t_2} \|u(\cdot,t)\|_X <\infty$. A similar definition holds for $L^\tp(t_1,t_2;X)$. 

Given an open box $\mathcal{O}=(a_1,b_1)\times (a_2,b_2)\times \cdots \times (a_N,b_N)$, we define the parabolic space $\mathcal{O}_{t_1,t_2}$ to be $\mathcal{O} \times [t_1,t_2]$.  Its parabolic boundary is given by 
$\p_{\pb}\mathcal{O}_{t_1,t_2} = (\overline{\mathcal{O}}\times \{t_1\}) \cup (\p \mathcal{O}\times (t_1,t_2])$. 
 
Finally, recall that if $G$ is a Carnot group with homogeneous dimension $\mathcal{Q}$, then $G \times \mathbb{R}$ is again a Carnot group of homogeneous dimension $\mathcal{Q} + 1$ where we have added an extra vector field $\frac{\partial}{\partial t}$ to the first layer of the grading.  This allows us to give meaning to notations such as $W^{1,2}(\mathcal{O}_{t_1,t_2})$ and $\ccs(\mathcal{O}_{t_1,t_2})$ where we consider $\nabla_0u$ to be $\left(X_1u, X_2u, \ldots, X_{n_1}u, \frac{\partial u}{\partial t}\right)$.
\section{Parabolic Jets and Viscosity Solutions}
\subsection{Parabolic Jets}
In this subsection, we recall the definitions of the parabolic jets, as given in \cite{BM}, but included here for completeness.  We define the parabolic superjet of $u(p,t)$ at the point $(p_0,t_0) \in \mathcal{O}_{t_1,t_2}$, denoted $P^{2,+}u(p_0,t_0)$, by using triples $(a,\eta,X) \in \mathbb{R} \times V_1\oplus V_2 \times S^{n_1}$ so that $(a,\eta,X) \in P^{2,+}u(p_0,t_0)$ if 
\begin{eqnarray*}
u(p,t) & \leq & u(p_0,t_0) + a(t-t_0) + \ip{\eta}{\widehat{p_0^{-1}\cdot p}}+\frac{1}{2}\ip{X\overline{p_0^{-1}\cdot p}}{\overline{p_0^{-1}\cdot p}}\\
 & & + o(|t-t_0|+|p_0^{-1}\cdot p|^2)\ \ \textmd{as}\ \ (p,t)\rightarrow (p_0,t_0).
\end{eqnarray*}
We recall that $S^k$ is the set of $k\times k$ symmetric matrices and $n_i=\dim V_i$. We define $\overline{p_0^{-1}\cdot p}$ as the first $n_1$ coordinates of 
$p_0^{-1}\cdot p$ and $\widehat{p_0^{-1}\cdot p}$ as the first $n_1+n_2$ coordinates of 
$p_0^{-1}\cdot p$.  
This definition is an extension of the superjet definition for subparabolic equations in the Heisenberg group \cite{B:HP}.  We define the subjet $P^{2,-}u(p_0,t_0)$ by 
$$P^{2,-}u(p_0,t_0)=-P^{2,+}(-u)(p_0,t_0).$$  
We define the set theoretic closure of the superjet, denoted $\overline{P}^{2,+}u(p_0,t_0)$, by requiring $(a,\eta,X) \in \overline{P}^{2,+}u(p_0,t_0)$ exactly when there is a sequence \\
$(a_n,p_n,t_n,u(p_n,t_n),\eta_n,X_n)\to (a,p_0,t_0,u(p_0,t_0),\eta,X)$ with the triple \\ $(a_n,\eta_n,X_n)\in P^{2,+}u(p_n,t_n)$. A similar definition holds for the closure of the subjet.

We may also define jets using appropriate test functions. Given a function $u:\mathcal{O}_{t_1,t_2}\to \mathbb{R}$
we consider the set $\mathcal{A}u(p_0,t_0)$ given by
\begin{equation*}
\mathcal{A}u(p_0,t_0) = \{\phi \in \ccs(\mathcal{O}_{t_1,t_2}):u(p,t)-\phi(p,t)\leq u(p_0,t_0)-\phi(p_0,t_0)=0\  \forall (p,t)\in \mathcal{O}_{t_1,t_2}\}.
\end{equation*}
consisting of all test functions that touch $u$ from above at $(p_0,t_0)$. We define 
the set of all test functions that touch from below, denoted $\mathcal{B}u(p_0,t_0)$,
similarly.

The following lemma relates the test functions to jets. The proof is identical to Lemma 3.1 in \cite{B:HP}, but uses the (smooth) gauge $\mathcal{N}(p)$ instead of Euclidean distance.
\begin{lemma}
$$P^{2,+}u(p_0,t_0)=\{(\phi_t(p_0,t_0),\nabla \phi(p_0,t_0), (D^2\phi(p_0,t_0))^\star): \phi \in \mathcal{A}u(p_0,t_0)\}.$$ 
\end{lemma}

\subsection{Jet Twisting}
We recall that the set $V_1=\span \{X_1,X_2,\ldots, X_{n_1}\}$ and notationally, we will always denote $n_1$ by $n$. The vectors $X_i$ at the point $p\in G$ can be written as 
$$X_i(p)=\sum_{j=1}^N a_{ij}(p)\frac{\p}{\p x_j}$$
forming the $n\times N$ matrix $\mathbb{A}$ with smooth entries $\mathbb{A}_{ij}=a_{ij}(p)$. By linear independence of the $X_i$, $\mathbb{A}$ has rank $n$. Similarly, $$Y_i(p)=
\sum_{j=1}^N b_{ij}(p)\frac{\p}{\p x_j}$$ forming the $n_2 \times N$ matrix $\mathbb{B}$ with smooth entries $\mathbb{B}_{ij}=b_{ij}$. The matrix $\mathbb{B}$ has rank $n_2$. 
The following lemma differs from \cite[Corollary 3.2]{B:MP} only in that there is now a parabolic term. This term however, does not need to be twisted. The proof is then identical, as only the space terms need twisting.  
\begin{lemma}\label{twist}
Let $(a,\eta, X) \in \overline{P}^{2,+}_{\eucl}u(p,t)$. (Recall that $(\eta,X) \in \mathbb{R}^N \times S^N$.)  Then $$(a, \mathbb{A} \cdot \eta \oplus \mathbb{B} \cdot \eta, \  
\mathbb{A}X\mathbb{A}^T+\mathbb{M}) \in \overline{P}^{2,+}u(p,t).$$ Here the entries of the (symmetric) matrix $\mathbb{M}$ are given by 
\begin{eqnarray*}
 \mathbb{M}_{ij} & = &\begin{cases}
  \D \sum_{k=1}^N\D\sum_{l=1}^N \bigg(a_{il}(p)\frac{\D\p}{\D\p x_l}a_{jk}(p)+a_{jl}(p)\frac{\D\p a_{ik}}{\D\p x_l}(p)\bigg)\eta_k     & i \neq j, \\
    & \\
  \D\sum_{k=1}^N\D\sum_{l=1}^N a_{il}(p)\frac{\D\p a_{ik}}{\D\p x_l}(p)\eta_k    & i = j.
\end{cases}
\end{eqnarray*}
\end{lemma}
\subsection{Viscosity Solutions}
We consider parabolic equations of the form 
\begin{equation}\label{main}
u_t+F(t,p,u,\nabla_1 u,(D^2u)^{\star})=0
\end{equation}
for continuous and proper $F:[0,T]\times G \times \mathbb{R} \times g \times S^{n} \to \mathbb{R}$.  \cite{CIL:UGTVS}
We recall that $S^{n}$ is the set of $n \times n$ symmetric matrices (where $\dim V_1=n$) and the derivatives  $\nabla_1 u$ and $(D^2u)^{\star}$ are taken in the space variable $p$.  
We then use the jets to define subsolutions and supersolutions to Equation \eqref{main} in the usual way.
\begin{definition}\label{1}
Let $(p_0,t_0)\in \mathcal{O}_{t_1,t_2}$ be as above.  The upper semicontinuous function $u$ is a \emph{parabolic viscosity subsolution} in $\mathcal{O}_{t_1,t_2}$ if for all $(p_0,t_0) \in \mathcal{O}_{t_1,t_2}$ we have $(a,\eta,X) \in \overline{P}^{2,+}u(p_0,t_0)$ produces 
$$a+F(t_0,p_0,u(p_0,t_0),\eta,X)\leq 0.$$
A lower semicontinuous function $u$ is a \emph{parabolic viscosity supersolution} in $\mathcal{O}_{t_1,t_2}$ if for all $(p_0,t_0) \in \mathcal{O}_{t_1,t_2}$ we have $(b,\nu,Y) \in \overline{P}^{2,-}u(p_0,t_0)$ produces 
$$b+F(t_0,p_0,u(p_0,t_0),\nu,Y)\geq 0.$$
A continuous function $u$ is a \emph{parabolic viscosity solution} in $\mathcal{O}_{t_1,t_2}$ if it is both a parabolic viscosity subsolution and parabolic viscosity supersolution.
\end{definition}
\begin{remark}
In the special case when $F(t,p,u,\nabla_1 u,(D^2u)^{\star})=F^h_\infty(\nabla_0 u,(D^2u)^{\star})=-\Delta^h_{\infty}u$, for $h\geq 3$, we use the terms ``parabolic viscosity h-infinite supersolution", etc. 
\end{remark}

In the case when $1\leq h<3$, the definition above is insufficient due to the singularity occurring when the horizontal gradient vanishes.  Therefore, following \cite{JK:JK} and \cite{PV}, we define viscosity solutions to
Equation \eqref{eqmain} when $1\leq h<3$ as follows: 
\begin{definition}\label{2}
Let $\mathcal{O}_{t_1,t_2}$ be as above. A lower semicontinuous function 
$v:\mathcal{O}_{t_1,t_2}\to\mathbb{R}$ is a \emph{parabolic viscosity h-infinite supersolution} of $u_t-\Delta^h_\infty u=0$ if whenever $(p_0,t_0)\in \mathcal{O}_{t_1,t_2}$ and $\phi\in \mathcal{B}u(p_0,t_0)$, we have 
\begin{eqnarray*}\left\{\begin{array}{cc}
\phi_t(p_0,t_0)-\Delta^h_\infty \phi(p_0,t_0) \geq 0 & \textmd{when}\ \ \nabla_0 \phi(p_0,t_0) \neq 0 \\ \mbox{} & \mbox{} \\
\phi_t(p_0,t_0)-\D \min_{\|\eta\|=1}\ip{(D^2\phi)^\star(p_0,t_0)\;\eta}{\eta} \geq 0 & \textmd{when}\ \ \nabla_0 \phi(p_0,t_0) = 0 \ \textmd{and\ \ } h=1 \\
\phi_t(p_0,t_0) \geq 0 & \textmd{when}\ \ \nabla_0 \phi(p_0,t_0) = 0 \ \textmd{and\ \ } 1<h<3 
\end{array}\right.
\end{eqnarray*}
An upper semicontinuous function 
$u:\mathcal{O}_{t_1,t_2}\to\mathbb{R}$ is a \emph{parabolic viscosity h-infinite subsolution} of $u_t-\Delta^h_\infty u=0$ if whenever $(p_0,t_0)\in \mathcal{O}_{t_1,t_2}$ and $\phi\in \mathcal{A}u(p_0,t_0)$, we have 
\begin{eqnarray*}\left\{\begin{array}{cc}
\phi_t(p_0,t_0)-\Delta^h_\infty \phi(p_0,t_0) \leq 0 & \textmd{when}\ \ \nabla_0 \phi(p_0,t_0) \neq 0 \\  \mbox{} & \mbox{} \\
\phi_t(p_0,t_0)-\D \max_{\|\eta\|=1}\ip{(D^2\phi)^\star(p_0,t_0)\;\eta}{\eta} \leq 0 & \textmd{when}\ \ \nabla_0 \phi(p_0,t_0) = 0 \ \textmd{and\ \ } h=1 \\
\phi_t(p_0,t_0) \leq 0 & \textmd{when}\ \ \nabla_0 \phi(p_0,t_0) = 0 \ \textmd{and\ \ } 1<h<3 
\end{array}\right.
\end{eqnarray*}
A continuous function is a \emph{parabolic viscosity h-infinite solution} if it is both a parabolic viscosity h-infinite subsolution and parabolic viscosity h-infinite subsolution. 
\end{definition}
\begin{remark}
When $1<h<3$, we can actually consider the continuous operator 
\begin{eqnarray}\label{relax}
F^h_\infty(\nabla_0 u,(D^2u)^{\star}) = 
\left\{\begin{array}{cl}
-\|\nabla_0u\|^{h-3}\ip{(D^2u)^{\star}\nabla_0u}{\nabla_0u}=-\Delta^h_\infty u & \nabla_0u \neq 0 \\
0 & \nabla_0u = 0. \end{array}\right.
\end{eqnarray}
Definitions \ref{1} and {2} would then agree. (cf. \cite{PV})
\end{remark}

We also wish to define what \cite{J:PD} refers to as parabolic viscosity solutions. We first need to consider the set $$\mathcal{A}^-u(p_0,t_0)=\{\phi \in \mathcal{C}^2(\mathcal{O}_{t_1,t_2}): u(p,t)-\phi(p,t) \leq u(p_0,t_0)-\phi(p_0,t_0)=0\  \textmd{for}\ p \neq p_0, t < t_0\}$$ 
consisting of all functions that touch from above only when $t<t_0$.  
Note that this set is larger than $\mathcal{A}u$ and corresponds physically to the past alone playing a role in determining the present. We define $\mathcal{B}^-u(p_0,t_0)$ similarly.
We then have the following definition.

\begin{definition}
An upper semicontinuous function $u$ on $\mathcal{O}_{t_1,t_2}$ is a \emph{past parabolic viscosity subsolution} in $\mathcal{O}_{t_1,t_2}$ if $\phi\in \mathcal{A}^-u(p_0,t_0)$ produces 
$$\phi_t(p_0,t_0)+F(t_0,p_0,u(p_0,t_0),\nabla_1 \phi(p_0,t_0),(D^2\phi(p_0,t_0))^{\star}) \leq 0.$$ 
An lower semicontinuous function $u$ on $\mathcal{O}_{t_1,t_2}$ is a \emph{past parabolic viscosity supersolution} in $\mathcal{O}_{t_1,t_2}$ if $\phi\in \mathcal{B}^-u(p_0,t_0)$ produces 
$$\phi_t(p_0,t_0)+F(t_0,p_0,u(p_0,t_0),\nabla_1 \phi(p_0,t_0),(D^2\phi(p_0,t_0))^{\star}) \geq 0.$$ 
A continuous function is a \emph{past parabolic viscosity solution} if it is both a past parabolic viscosity supersolution and subsolution.
\end{definition} 
We have the following proposition whose proof is obvious. 
\begin{prop}\label{onedir}
Past parabolic viscosity sub(super-)solutions are parabolic viscosity sub(super-)solutions. In particular, past parabolic viscosity h-infinite sub(super-)solutions are parabolic viscosity h-infinite subsub(super-)solutions for $h \geq 1$. 
\end{prop}

\subsection{The Carnot Parabolic Maximum Principle}
In this subsection, we recall the Carnot Parabolic Maximum Principle and key corollaries, as proved in \cite{BM}.  
\begin{lemma}[Carnot Parabolic Maximum Principle]\label{cpmp}
Let $u$ be a viscosity subsolution to Equation \eqref{main} and $v$ be
a viscosity supersolution to Equation \eqref{main} in the bounded parabolic set $\Omega \times (0,T)$ where $\Omega$ is a (bounded) domain and let $\tau$ be a positive real parameter. Let $\phi(p,q,t)=\varphi(p\cdot q^{-1},t)$ be a $C^2$ function in the space variables $p$ and $q$ and a $C^1$ function in $t$.  
Suppose the local maximum 
\begin{equation}\label{Mdef}
M_\tau \equiv \max_{\overline{\Omega}\times \overline{\Omega}\times [0,T]}\{u(p,t)-v(q,t)-\tau\phi(p,q,t)\}
\end{equation}
 occurs at the interior point $(p_\tau, q_\tau, t_\tau)$ of the parabolic set $\Omega\times \Omega \times (0,T)$.
Define the $n \times n$ matrix $W$ by 
\begin{eqnarray*}
W_{ij}=X_i(p)X_j(q)\phi(p_\tau,q_\tau, t_\tau).
\end{eqnarray*}
Let the $2n \times 2n$ matrix $\mathfrak{W}$ be given by
\begin{eqnarray}
\mathfrak{W}=\begin{pmatrix}
    0  &  \frac{1}{2}(W-W^T)  \\
    \frac{1}{2}(W^T-W)   &  0
\end{pmatrix}
\end{eqnarray}
and let the matrix  $\mathcal{W}\in S^{2N}$ be given by 
\begin{eqnarray}\label{Adef}
\mathcal{W}=\begin{pmatrix}
  D^2_{pp}\phi(p_\tau,q_\tau, t_\tau)    &  D^2_{pq}\phi(p_\tau,q_\tau, t_\tau)  \\
   & \\
   D^2_{qp}\phi(p_\tau,q_\tau, t_\tau)   & D^2_{qq}\phi(p_\tau,q_\tau, t_\tau) 
\end{pmatrix}_.
\end{eqnarray}
Suppose $$\lim_{\tau \rightarrow \infty}\tau\phi(p_\tau,q_\tau, t_\tau)=0.$$ Then for each $\tau>0$, 
there exists  real numbers $a_1$ and $a_2$, symmetric matrices $\mathcal{X}_{\tau}$ and
$\mathcal{Y}_{\tau}$ and
vector $\Upsilon_{\tau} \in V_1 \oplus V_2$, namely 
$\Upsilon_{\tau}=\nabla_{1}(p)\phi(p_{\tau}, q_{\tau}, t_\tau)$, 
so that the following hold:
\begin{enumerate}
\item [A)]  $(a_1,\tau \Upsilon_{\tau},\mathcal{X}_{\tau}) \in \overline{P}^{2,+}u(p_\tau,t_\tau)$ and $(a_2,\tau \Upsilon_{\tau},\mathcal{Y}_{\tau}) \in \overline{P}^{2,-}v(q_\tau,t_\tau).$
\item [B)] $a_1-a_2=\phi_t(p_\tau, q_\tau, t_\tau).$
\item [C)] For any vectors $\xi, \epsilon \in V_1$, we have 
\begin{eqnarray}
 \ip{\mathcal{X}_{\tau}\xi}{\xi} - \ip{\mathcal{Y}_{\tau}\epsilon}{\epsilon} & \leq & 
\tau \ip{(D^2_p\phi)^\star(p_\tau, q_{\tau}, t_\tau) (\xi-\epsilon)}{(\xi-\epsilon)}+\tau\ip{\mathfrak{W}(\xi\oplus\epsilon)}{(\xi\oplus\epsilon)}\nonumber\\ 
& & \mbox{}+\tau \|\mathcal{W}\|^2\|\mathbb{A}(\hat{p})^T\xi\oplus\mathbb{A}(\hat{q})^T\epsilon\|^2. \label{fake}
\end{eqnarray}
In particular, 
 \begin{equation}
 \ip{\mathcal{X}_{\tau}\xi}{\xi} - \ip{\mathcal{Y}_{\tau}\xi}{\xi}\lesssim 
\tau \|\mathcal{W}\|^2\|\xi\|^2.\label{real}
\end{equation}
\end{enumerate}
\end{lemma}
\begin{corollary}\label{cpmpcor}
Let $\phi(p, q, t) = \phi(p,q)=\varphi(p\cdot q^{-1})$ be independent of $t$ and a non-negative function. Suppose $\phi(p,q)=0$ exactly when $p=q$. Then
$$\lim_{\tau\to\infty} \tau \phi(p_\tau, q_\tau)=0.$$ In particular, if 
\begin{equation}\label{phidef}
\phi(p, q, t) = \frac{1}{m}\sum_{i=1}^{N}\big((p\cdot q^{-1})_i \big)^m
\end{equation}
for some \textbf{even} integer $m\geq4$ where $(p\cdot q^{-1})_i$ is the $i$-th component of the Carnot group multiplication group law, then for the vector $\Upsilon_\tau$ and matrices $\mathcal{X}_\tau, \mathcal{Y}_\tau$,  from the Lemma,  we have 
\begin{enumerate}
\item [A)] $(a_1,\tau \Upsilon_\tau,\mathcal{X}_{\tau}) \in \overline{P}^{2,+}u(p_\tau,t_\tau)$ and $(a_1,\tau \Upsilon_\tau,\mathcal{Y}_{\tau}) \in \overline{P}^{2,-}v(q_\tau,t_\tau).$
\item [B)] The vector $\Upsilon_\tau$ satisfies $$\| \Upsilon_\tau\| \sim \phi(p_\tau, q_\tau)^{\frac{m-1}{m}}.$$
\item [C)] For any fixed vector $\xi \in V_1$, we have 
 \begin{equation}
 \ip{\mathcal{X}_{\tau}\xi}{\xi} - \ip{\mathcal{Y}_{\tau}\xi}{\xi}\lesssim 
\tau \|\mathcal{W}\|^2\|\xi\|^2 \lesssim \tau(\phi(p_\tau, q_\tau))^{\frac{2m-4}{m}}\|\xi\|^2.\label{real2}
\end{equation}
\end{enumerate}
\end{corollary}

\section{Uniqueness of viscosity solutions}
We wish to formulate a comparison principle for the following problem.
\begin{problem}\label{mainprob}
Let $h \geq 1$. Let $\Omega$ be a bounded domain and let $\Omega_T=\Omega\times [0,T)$. 
Let $\psi \in C(\overline{\Omega})$ and $g \in C(\overline{\Omega_T})$. We consider the following boundary and initial value problem:
\begin{eqnarray}\label{problem}
\left\{\begin{array}{rl}
 u_t+F^h_\infty(\nabla_0 u, (D^2u)^\star) = 0 & \textmd{in}\ \ \Omega \times (0,T)  \hspace{.8in}(E)\\
u(p,t)=g(p,t) & p \in \partial \Omega,\ t \in [0,T)\hspace{.5in} (BC)\\
u(p,0) = \psi(p) & p \in \overline{\Omega}\hspace{1.25in}(IC)
\end{array}\right.
\end{eqnarray}
We also adopt the definition that a subsolution $u(p,t)$ to Problem \ref{problem} is a viscosity subsolution to (E), $u(p,t) \leq g(p,t)$ on $\partial \Omega$ with $0 \leq t < T$ and $u(p,0) \leq \psi(p)$ on $\overline{\Omega}$.  Supersolutions and solutions are defined in an analogous matter. 
\end{problem}
Because our solution $u$ will be continuous, we offer the following remark: 
\begin{remark}
The functions $\psi$ and $g$ may be replaced by one function $g\in C(\overline{\Omega_T})$. This combines conditions (E) and (BC) into one condition 
\begin{equation}
u(p,t)=g(p,t),\ \ \  (p,t) \in \partial_{\pb} \Omega_T \hspace{.5in} (IBC)
\end{equation}
\end{remark} 
\begin{thm}\label{pinf}
Let $\Omega$ be a bounded domain in $G$ and let $h\geq 1$.   If $u$ is a parabolic viscosity subsolution and $v$ a parabolic viscosity supersolution to Problem \eqref{problem} then $u \leq v$ on $\Omega_T\equiv\Omega\times [0,T) $.
\end{thm}
\begin{proof}
Our proof follows that of \cite[Thm. 8.2]{CIL:UGTVS} and so we discuss only the main parts. 

For $\varepsilon > 0$, we substitute $\tilde{u}=u-\frac{\D\varepsilon}{T-t}$ for $u$ and prove the theorem for 
\begin{eqnarray}
u_t+F^h_\infty(\nabla_0 u,(D^2u)^\star) \leq -\frac{\varepsilon}{T^2} < 0 \\ 
\lim_{t \uparrow T}u(p,t) = -\infty \ \  \textmd{uniformly on }\ \ \overline{\Omega}
\end{eqnarray}
and take limits to obtain the desired result.
Assume the maximum occurs at $(p_0,t_0)\in \Omega \times (0,T)$ with $$u(p_0,t_0)-v(p_0,t_0)= \delta >0.$$
\textbf{Case 1: $\boldsymbol{h>1.}$} \\
Let $H\geq h+3$ be an even number. As in Equation \eqref{phidef}, we let $$\phi(p, q) = \frac{1}{H}\sum_{i=1}^{N}\big((p\cdot q^{-1})_i \big)^H$$
where $(p\cdot q^{-1})_i$ is the $i$-th component of the Carnot group multiplication group law. Let 
$$M_\tau=u(p_\tau,t_\tau)-v(q_\tau,t_\tau)-\tau\phi(p_\tau,q_\tau)$$ with $(p_\tau,q_\tau,t_\tau)$ the maximum point in $\overline{\Omega} \times \overline{\Omega} \times [0,T)$ of 
$u(p,t)-v(q,t)-\tau \phi(p,q)$. 

If $t_\tau=0$, we have 
$$0 < \delta \leq M_\tau \leq \sup_{\overline{\Omega}\times\overline{\Omega}}(\psi(p)-\psi(q)-\tau \phi(p,q))$$ leading to a contradiction for large $\tau$.  We therefore conclude $t_\tau >0$ for large $\tau$.  Since $u \leq v$ on $\partial \Omega \times [0,T)$ by Equation (BC) of Problem \eqref{problem}, we conclude that for large $\tau$, we have $(p_\tau,q_\tau,t_\tau)$ is an interior point. That is, 
$(p_\tau,q_\tau,t_\tau) \in \Omega \times \Omega \times (0,T)$.
Using  Corollary \ref{cpmpcor} Property A,  we obtain
\begin{eqnarray*}
(a,\tau \Upsilon(p_\tau,q_\tau), \mathcal{X}_\tau) & \in & \overline{P}^{2,+}u(p_\tau,t_\tau) \\
(a,\tau \Upsilon(p_\tau,q_\tau), \mathcal{Y}_\tau) & \in & \overline{P}^{2,-}v(q_\tau,t_\tau)
\end{eqnarray*}  
satisfying the equations 
\begin{eqnarray*}
a+F^h_\infty(\tau\Upsilon(p_\tau,q_\tau), \mathcal{X}_\tau) & \leq & -\frac{\varepsilon}{T^2} \\
a+F^h_\infty(\tau\Upsilon(p_\tau,q_\tau), \mathcal{Y}_\tau) & \geq & 0.
\end{eqnarray*}

If there is a subsequence $\{p_\tau,q_\tau\}_{\tau>0}$ such that $p_\tau\neq q_\tau$,  we subtract, and using Corollary \ref{cpmpcor}, we have 
\begin{eqnarray} \label{festimate}
0 < \frac{\varepsilon}{T^2}  & \leq & (\tau\Upsilon(p_\tau,q_\tau))^{h-3} \tau^2\bigg(\ip{\mathcal{X}_\tau\Upsilon(p_\tau,q_\tau)}{\Upsilon(p_\tau,q_\tau)}-\ip{\mathcal{Y}_\tau\Upsilon(p_\tau,q_\tau)}{\Upsilon(p_\tau,q_\tau)}\bigg)\nonumber \\
 & \lesssim & \tau^h\big( \varphi(p_\tau,q_\tau)^{\frac{H-1}{H}}\big)^{h-3}(\varphi(p_\tau,q_\tau))^{\frac{2H-4}{H}} (\varphi(p_\tau,q_\tau))^{\frac{2H-2}{H}} \\
& = & \tau^h(\varphi(p_\tau,q_\tau))^{\frac{Hh+H-h-3}{H}}= (\tau\varphi(p_\tau,q_\tau))^h 
\varphi(p_\tau,q_\tau)^{\frac{H-h-3}{H}}.
\end{eqnarray}
Because $H>h+3$, we arrive at a contradiction as $\tau \rightarrow \infty$. 

If we have  $p_\tau = q_\tau$, we arrive at a contradiction since $$F^h_\infty(\tau\Upsilon(p_\tau,q_\tau), \mathcal{X}_\tau)=F^h_\infty(\tau\Upsilon(p_\tau,q_\tau), \mathcal{Y}_\tau)=0.$$ 

\noindent\textbf{Case 2: $\boldsymbol{h=1.}$} \\ 
We follow the proof of Theorem 3.1 in \cite{JK:JK}.  We let 
$$\varphi(p, q, t, s ) = \frac{1}{4}\sum_{i=1}^{N}\big((p\cdot q^{-1})_i \big)^4+\frac{1}{2}(t-s)^2$$
and let  $(p_\tau,q_\tau,t_\tau,s_\tau)$ be the maximum of $$u(p,t)-v(q,s)-\tau \phi(p, q,t ,s)$$ 
Again, for large $\tau$, this point is an interior point. If we have a sequence where $p_\tau \neq q_\tau$, then  Lemma \ref{twist} yields 
\begin{eqnarray*}
(\tau (t_\tau-s_\tau),\tau \Upsilon(p_\tau,q_\tau), \mathcal{X}_\tau) & \in & \overline{P}^{2,+}u(p_\tau,t_\tau) \\
(\tau (t_\tau-s_\tau),\tau \Upsilon(p_\tau,q_\tau), \mathcal{Y}_\tau) & \in & \overline{P}^{2,-}v(q_\tau,s_\tau)
\end{eqnarray*}  
satisfying the equations 
\begin{eqnarray*}
\tau (t_\tau-s_\tau)+F^h_\infty(\tau\Upsilon(p_\tau,q_\tau), \mathcal{X}_\tau) & \leq & -\frac{\varepsilon}{T^2} \\
\tau (t_\tau-s_\tau)+F^h_\infty(\tau\Upsilon(p_\tau,q_\tau), \mathcal{Y}_\tau) & \geq & 0.
\end{eqnarray*}
As in the first case, we subtract to obtain
\begin{eqnarray*} 
0 < \frac{\varepsilon}{T^2}  & \leq & (\tau\Upsilon(p_\tau,q_\tau))^{-2} \tau^2\bigg(\ip{\mathcal{X}_\tau\Upsilon(p_\tau,q_\tau)}{\Upsilon(p_\tau,q_\tau)}-\ip{\mathcal{Y}_\tau\Upsilon(p_\tau,q_\tau)}{\Upsilon(p_\tau,q_\tau)}\bigg) \\
 & \lesssim &  \varphi(p_\tau,q_\tau)^{-\frac{3}{2}}(\tau\varphi(p_\tau,q_\tau) \varphi(p_\tau,q_\tau)^{\frac{3}{2}}) = \tau\varphi(p_\tau,q_\tau).
\end{eqnarray*}
We arrive at a contradiction as $\tau\to\infty$. 

If $p_\tau = q_\tau$, then 
$v(q,s)-\beta^v(q,s)$ has a local minimum at $(q_\tau,s_\tau)$ where 
$$\beta^v(q,s)=- \frac{\tau}{4}\sum_{i=1}^{N}\big((p_\tau\cdot q^{-1})_i \big)^4-\frac{\tau}{2}(t_\tau-s)^2.$$  We then have 
$$0<\varepsilon(T-s_\tau)^{-2} \leq \beta^v_s(q_\tau,s_\tau)-\min_{\|\eta\|=1}\ip{(D^2\beta^v)^\star (q_\tau,s_\tau)\;\eta}{\eta}.$$

Similarly, $u(p,t)-\beta^u(p,t)$ has a local maximum at $(p_\tau,t_\tau)$ where 
$$\beta^u(p,t)=\frac{\tau}{4}\sum_{i=1}^{N}\big((p\cdot q_\tau^{-1})_i \big)^4+\frac{\tau}{2}(t-s_\tau)^2.$$
We then have 
$$0 \geq \beta^u_t(p_\tau,t_\tau)-\max_{\|\eta\|=1}\ip{(D^2\beta^u)^\star (p_\tau,t_\tau)\;\eta}{\eta}$$
and subtraction gives us 
\begin{eqnarray*}
0<\varepsilon(T-s_\tau)^{-2} & \leq &\max_{\|\eta\|=1}\ip{(D^2\beta^u)^\star (p_\tau,t_\tau)\;\eta}{\eta}-\min_{\|\eta\|=1}\ip{(D^2\beta^v)^\star (q_\tau,s_\tau)\;\eta}{\eta} \\
& & \mbox{}+ \beta^v_s(q_\tau,s_\tau)-\beta^u_t(p_\tau,t_\tau) \\
 & = & \tau\max_{\|\eta\|=1}\ip{(D_{pp}^2\varphi(p\cdot q_\tau^{-1}))^\star (p_\tau,t_\tau)\;\eta}{\eta} \\
 & & \mbox{}-\tau\min_{\|\eta\|=1}\ip{(D^2_{qq}\varphi(p_\tau\cdot q^{-1}))^\star (q_\tau,s_\tau)\;\eta}{\eta} \\
 & & \mbox{}+ \tau(t_\tau-s_\tau)- \tau(t_\tau-s_\tau) \\
 & = & 0.
\end{eqnarray*}
Here, the last equality comes from the fact that $p_\tau=q_\tau$ and the definition of $\varphi(p\cdot q^{-1})$.
\end{proof}

The comparison principle has the following consequences concerning properties of solutions: 
\begin{corollary}\label{h-infinite}
Let $h\geq 1$. 
The past parabolic viscosity h-infinite solutions are exactly the parabolic viscosity h-infinite solutions. 
\end{corollary}
\begin{proof}
By Proposition \ref{onedir}, past parabolic viscosity h-infinite sub(super-)solutions are parabolic viscosity h-infinite  sub(super-)solutions. To prove the converse, we will follow the proof of the subsolution case found in 
\cite{J:PD}, highlighting the main details.  Assume that $u$ is not a past parabolic viscosity h-infinite subsolution. 
Let $\phi \in \mathcal{A}^-u(p_0,t_0)$ have the property that 
$$\phi_t(p_0,t_0)-\Delta^h_\infty  \phi(p_0,t_0) \geq \epsilon > 0$$ for a small parameter $\epsilon$.  We may assume $p_0$ is the origin.
Let $r > 0$ and define $S_r= B_{\mathcal{N}}(r) \times (t_0-r,t_0)$ and let $\partial S_r$ be its parabolic boundary. Then the function $$\tilde{\phi}_r(p,t)= \phi(p,t)+(t_0-t)^{8l!}-r^{8l!}+(\mathcal{N}(p))^{8l!}$$ is a classical supersolution for sufficiently small $r$.  We then observe that $u \leq \tilde{\phi}_r$ on $\partial S_r$ but $u(0,t_0) > \tilde{\phi}(0,t_0)$. Thus, the comparison prinicple, Theorem \ref{pinf}, does not hold.  Thus, $u$ is not a parabolic viscosity h-infinite subsolution.  The supersolution case is identical and omitted. 
\end{proof}
The following corollary has a proof similar to \cite[Lemma 3.2]{JK:JK}. 
\begin{corollary}\label{moddef}
Let $u: \Omega_T \to \mathbb{R}$ be upper semicontinuous. Let $\phi\in \mathcal{A}u(p_0,t_0)$. If 
\begin{eqnarray}\label{modhyp}
\left\{\begin{array}{cc}
\phi_t(p_0,t_0)-\Delta^1_\infty \phi(p_0,t_0) \leq 0 & \textmd{when}\ \ \nabla_0 \phi(p_0,t_0) \neq 0 \\  
\phi_t(p_0,t_0) \leq 0 & \textmd{when}\ \ \nabla_0 \phi(p_0,t_0) = 0, (D^2\phi)^\star(p_0,t_0)=0
\end{array}\right.
\end{eqnarray}
then $u$	is a viscosity subsolution to (E) of Problem \eqref{problem}. 
\end{corollary}

We also have the following function estimates with respect to boundary data. 
\begin{corollary}
Let $h\geq1$. Let $g_1, g_2 \in C(\overline{\Omega_T})$ and $u_1$, $u_2$ be parabolic viscosity  solutions to Equation \ref{problem} 
with boundary data $g_1$ and $g_2$, respectively. Then
$$\sup_{(p,t)\in\Omega_T}|u_1(p,t)-u_2(p,t)|\leq \sup_{(p,t)\in\partial_{\pb}\Omega_T}|g_1(p,t)-g_2(p,t)|.  $$
\end{corollary}
\begin{proof}
The function $u^+(p,t)=u_2(p,t)+\sup_{(p,t)\in\partial_{\pb}\Omega_T}|g_1(p,t)-g_2(p,t)|$ is a parabolic viscosity supersolution with boundary data $g_1$ and the function $u^-(p,t)=u_2(p,t)-\sup_{(p,t)\in\partial_{\pb}\Omega_T}|g_1(p,t)-g_2(p,t)|$ is a parabolic viscosity subsolution with boundary data $g_1$. Moreover, $u^- \leq u_1 \leq u^+$ on $\partial_{\pb}\Omega_T$ and by Theorem \ref{pinf} 
$u^- \leq u_1 \leq u^+$ in $\Omega_T$.
\end{proof}
\begin{corollary}
Let $h\geq1$. Let $g\in C(\overline{\Omega_T})$. Then every parabolic viscosity solution to Problem \ref{problem} 
satisfies $$\sup_{(p,t)\in\Omega_T}|u(p,t)|\leq \sup_{(p,t)\in\partial_{\pb}\Omega_T}|g(p,t)|$$
\end{corollary}
\begin{proof}
The proof is similar to the previous corollary, but using the functions $u^\pm(p,t)=\pm\sup_{(p,t)\in\partial_{\pb}\Omega_T}|g(p,t)|$ instead. 
\end{proof}

\section{Existence of Viscosity Solutions}
\subsection{Parabolic Viscosity Infinite Solutions: The Continuity Case}  As above, we will focus on the equations of the form \eqref{main}
for continuous and proper $F:[0,T]\times G \times \mathbb{R} \times g \times S^{n_1} \to \mathbb{R}$ that possess a comparison principle such as Theorem \ref{pinf} or \cite[Thm. 3.6]{BM}.  We will use Perron's method combined with the Carnot Parabolic Maximum Principle to yield the desired existence theorem.  In particular, the following proofs are similar to those found in \cite[Chapter 2]{Giga} except that the Euclidean derivatives have been replaced with horizontal derivatives and the Euclidean norms have been replaced with the gauge norm.
\begin{lemma}\label{lem1}
Let $\mathcal{L}$ be a collection of parabolic viscosity supersolutions to \eqref{main} and let $u(p,t) = \inf\{v(p,t): v \in \mathcal{L}\}$.  If $u$ is finite in a dense subset of $\Omega_T=\Omega\times [0,T)$ then $u$ is a parabolic viscosity supersolution to \eqref{main}.  
\end{lemma}
\begin{proof}
First note that $u$ is lower semicontinous since every $v \in \mathcal{L}$ is.  Let $(p_0,t_0) \in \Omega_T$ and $\phi \in \mathcal{A}u(p_0,t_0)$.  Now let 
$$\psi(p,t) = \phi(p,t) - \left(d_\mathcal{N}(p_0,p)\right)^{2l!} - |t-t_0|^2$$
and notice that $\psi \in \mathcal{A}u(p_0,t_0)$.
Then
\begin{eqnarray*}
(u-\psi)(p,t) - \left(d_\mathcal{N}(p_0,p)\right)^{2l!} - |t-t_0|^2 &=& (u-\phi)(p,t) \\&\geq& (u-\phi)(p_0,t_0) \\&=& (u-\psi)(p_0,t_0) \\&=& 0
\end{eqnarray*}
yields
\begin{equation}\label{bound}
(u-\psi)(p,t) \geq \left(d_\mathcal{N}(p_0,p)\right)^{2l!} + |t-t_0|^2.
\end{equation}
Since $u$ is lower semicontinuous, there exists a sequence $\{(p_k,t_k)\}$ with $t_k < t_0$ converging to $(p_0,t_0)$ as $k \to \infty$ such that 
$$(u-\psi)(p_k,t_k) \to (u-\psi)(p_0,t_0) = 0.$$
Since $u(p,t) = \inf \left\{ v(p,t): v \in \mathcal{L}\right\}$, there exists a sequence $\left\{v_k\right\} \subset \mathcal{L}$ such that $v_k(p_k,t_k) < u(p_k,t_k) + 1/k$ for k = 1,2, \ldots.  Since $v_k \geq u$, $\eqref{bound}$ gives us
\begin{equation}\label{bound2}
(v_k - \psi)(p,t) \geq (u-\psi)(p,t) \geq\left(d_\mathcal{N}(p_0,p)\right)^{2l!} + |t-t_0|^2.
\end{equation}
Let $B \subset \Omega$ denote a compact neighborhood of $(p_0,t_0)$.  Since $v_k-\psi$ is lower semicontinuous, it attains a minimum in $B$ at a point $(q_k,s_k) \in B$.  Then by $\eqref{bound}$ and $\eqref{bound2}$ we have 
$$ (u-\psi)(p_k,t_k) + 1/k > (v_k - \psi)(p_k,t_k) \geq (v_k - \psi)(q_k,s_k) \geq \left(d_\mathcal{N}(p_0,q_k)\right)^{2l!} + |s_k-t_0|^2 \geq 0$$
for sufficiently large $k$ such that $(p_k,t_k) \in B$.  By the squeeze theorem, $(q_k,s_k) \to (p_0,t_0)$ as $k \to \infty$.
Let $\eta = \psi - \left(d_\mathcal{N}(q_k,p)\right)^{2l!} - |s_k-t|^2$.  Then $\eta \in \mathcal{A}v_k(q_k,s_k)$ and we have that 
$$\eta_t(q_k,s_k)+F(s_k, q_k, v_k(q_k,s_k),\nabla_1\eta(q_k,s_k), (D^2 \eta(q_k,s_k))^\star) \geq 0.$$
This implies 
$$\psi_t(q_k,s_k)+F(s_k, q_k, v_k(q_k,s_k),\nabla_1\psi(q_k,s_k), (D^2 \psi(s_k,s_k))^\star)\geq 0.$$
Letting $k \to \infty$ yields 
$$\phi_t(p_0,t_0)+F(t_0,p_0,u(p_0,t_0)\nabla_1\phi(p_0,t_0), (D^2\phi(p_0,t_0))^\star)\geq 0.$$
and that $u$ is a parabolic viscosity supersolution as desired.
\end{proof}
A similar argument yields the following.
\begin{lemma}\label{closedsup} 
Let $\mathcal{L}$ be a collection of parabolic viscosity subsolutions to \eqref{main} and let $u(p,t) = \sup\{v(p,t): v \in \mathcal{L}\}$.  If $u$ is finite in a dense subset of $\Omega_T$ then $u$ is a parabolic viscosity subsolution to \eqref{main}.  
\end{lemma}
For the following lemmas, we need to recall the following definition.
\begin{definition}
The \textit{upper and lower semi-continuous envelopes} of a function $u$ are given by 
$$u^*(p,t):= \lim_{r \downarrow 0}\sup \{u(q,s) : |q^{-1}p|_g + |s-t| \leq r\}$$
and
$$u_*(p,t):= \lim_{r \downarrow 0}\inf \{u(q,s) : |q^{-1}p|_g + |s-t| \leq r\},$$
respectively.
\end{definition}
\begin{lemma}\label{maximal}
Let $h$ be a parabolic viscosity supersolution to \eqref{main} in $\Omega_T$.  Let $\mathcal{S}$ be the collection of all parabolic viscosity subsolutions $v$ of \eqref{main} satisfying $v \leq h$.  If for $\hat{v} \in \mathcal{S}$, $\hat{v}_*$ is not a parabolic viscosity supersolution of $\eqref{main}$ then there is a function $w \in \mathcal{S}$ and a point $(p_0,t_0)$ such that $\hat{v}(p_0,t_0) < w(p_0,t_0)$.
\end{lemma}
\begin{proof}  Let $\hat{v} \in \mathcal{S}$ such that $\hat{v}_*$ is not a parabolic viscosity supersolution of \eqref{main}.  Then there exists $(\hat{p},\hat{t}) \in \Omega_T$ and $\phi \in \mathcal{A}\hat{v}_*(\hat{p},\hat{t})$ such that 
\begin{equation}\label{contradict}
\phi_t(p,t)+F(t,p,\hat{v}_*(p,t),\nabla_1\phi(p,t), (D^2\phi(p,t))^\star) > 0.
\end{equation}
Let
$$\psi(p,t) = \phi(p,t) - \left(d_\mathcal{N}(\hat{p},p)\right)^{2l!} - |t-\hat{t}|^2$$
and notice that $\psi \in \mathcal{A}\hat{v}_*(\hat{p},\hat{t})$.
As in Lemma \ref{lem1}, 
\begin{equation}\label{bounded1}
(\hat{v}_*-\psi)(p,t) \geq \left(d_\mathcal{N}(\hat{p},p)\right)^{2l!} + |t-\hat{t}|^2.
\end{equation}
Let $B$ denote a compact neighborhood of $(\hat{p},\hat{t})$ and let $$B_{k\epsilon} = B \cap \left\{ (p,t) : \left(d_\mathcal{N}(\hat{p},p)\right)^{2l!} \leq k\epsilon \text{ and } |t-\hat{t}|^2 \leq k\epsilon \right\}.$$  Since $\hat{v} \in \mathcal{S}$, we have that $\hat{v} \leq h$ and thus $\psi(\hat{p},\hat{t})=\hat{v}_*(\hat{p},\hat{t}) \leq \hat{v}(\hat{p},\hat{t})  \leq h(\hat{p},\hat{t})$.  However, if $\psi(\hat{p},\hat{t}) = h(\hat{p},\hat{t})$, then $\psi \in \mathcal{A}h(\hat{p},\hat{t})$ and inequality $\eqref{contradict}$ would be contradictory.  Thus, 
$$\psi(\hat{p},\hat{t}) < h(\hat{p},\hat{t}).$$
Since $\psi$ is continuous and $h$ is lower semicontinuous, there exists $\epsilon > 0$ such that 
$$\psi(p,t) + 4\epsilon \leq h(p,t)$$
for $(p,t) \in B_{2\epsilon}$. Notice that $\psi + 4\epsilon$ is a subsolution of \eqref{main} on the interior of $B_{2\epsilon}$.  Further, by $\eqref{bounded1}$
\begin{equation}\label{bounded2}
\hat{v}(p,t) \geq \hat{v}_*(p,t) \geq \psi(p,t) + 4\epsilon \text{ for } (p,t) \in B_{2\epsilon} \backslash B_{\epsilon}.
\end{equation}
We now define $\omega$ by 
\begin{equation*}
\omega = \left\{ \begin{array}{ll}
\max\{\psi(p,t) + 4\epsilon,\hat{v}(p,t)\} & (p,t) \in B_\epsilon \\
\hat{v}(p,t) & (p,t) \in \Omega_T\backslash B_\epsilon
\end{array} \right.
\end{equation*}
But by $\eqref{bounded2}$
$$\omega(p,t) = \max\{\psi(p,t) + 4\epsilon,\hat{v}(p,t)\} \text{ for } (p,t) \in B_{2\epsilon},$$
not just for $(p,t) \in B_\epsilon$.  Then by Lemma \ref{closedsup}, $\omega$ is a subsolution in the interior of $B_{2\epsilon}$ 
and thus a subsolution in $\Omega_T$.  Therefore, $\omega \in \mathcal{S}$.  Since
$$0 = (\hat{v}_*-\psi)(\hat{p},\hat{t}) = \lim_{r \downarrow 0} \inf\left\{ (\hat{v} - \psi)(p,t) : (p,t) \in B_r \right\}$$
there is a point $(p_0,t_0) \in B_\epsilon$ that satisfies 
$$\hat{v}(p_0,t_0) - \psi(p_0,t_0) < 4\epsilon$$
which yields 
$$\hat{v}(p_0,t_0) < \psi(p_0,t_0) + 4\epsilon = \omega(p_0,t_0).$$
Thus, we have constructed $\omega \in \mathcal{S}$ that satisfies $\hat{v}(p_0,t_0) < \omega(p_0,t_0)$.
\end{proof}
We then have the following existence theorem concerning parabolic viscosity solutions.
\begin{thm}\label{exist}
Let $f$ be a parabolic viscosity subsolution to \eqref{main} and $g$ be a parabolic viscosity supersolution to \eqref{main} satisfying $f \leq g$ on $\Omega_T$ and $f_* = g^*$ on $\p_{\pb}\mathcal{O}_{0,T}$.  Then there is a parabolic viscosity solution $u$ to \eqref{main} satisfying $u \in C(\overline{O_T})$. Explicitly, there exists a unique parabolic viscosity infinite solution to Problem \ref{mainprob} when $h>1$. 
\end{thm}
\begin{proof}
Let 
$$S = \{\nu : \hspace{3pt} \nu \text{ is a parabolic viscosity subsolution to \eqref{main} in } \Omega_T \text{ with } \nu \leq g \text{ in } \Omega_T\}$$
and
$$ u(p,t) = \sup \{ \nu(p,t) : \hspace{3pt} \nu \in S\}.$$
Since $f \leq g$, the set $S$ is nonempty.  Notice that $f\leq u \leq g$ by construction.  By Lemma $\eqref{closedsup}$, $u$ is a parabolic viscosity subsolution.  Suppose $u_*$ is not a parabolic viscosity supersolution.  Then by Lemma \ref{maximal}, there exists a function $w \in S$ and a point $(p_0,t_0) \in \Omega_T$ such that $u(p_0,t_0) < w(p_0,t_0)$.  But this contradicts the definition of $u$ at $(p_0,t_0)$.  Thus $u_*$ is a parabolic viscosity supersolution. By our assumptions on $f$ and $g$ on $\p_{\pb}\mathcal{O}_{0,T}$, 
$$u=u^* \leq g^* = f_* \leq u_*$$
on $\p_{\pb}\mathcal{O}_{0,T}$.  Then by the (assumed) comparison principle,  $u \leq u_*$ on $\Omega_T$.  Thus we have $u$ is a parabolic viscosity solution such that $u \in C(\overline{O_T})$.  
\end{proof}
\subsection{The $h=1$ case}
We begin by recalling the definition of upper and lower relaxed limit of a function. \cite{CIL:UGTVS, Giga}. 
\begin{definition}
For $\varepsilon>0$, consider the function $h_\varepsilon:O_T\subset G\to\mathbb{R}$. The \emph{upper relaxed limit} $\overline{h}(p,t)$ and the lower relaxed limit $\underline{h}(p,t)$ are given by 
\begin{eqnarray*}
\overline{h}(p,t) & = & \limsup_{\hat{p}\to p,\hat{t}\to t, \varepsilon\to 0} h_\varepsilon(\hat{p}, \hat{t}) \\
 & = & \lim_{\varepsilon\to 0} \sup_{0<\delta<\varepsilon}\{h_\delta(\hat{p},\hat{t}): O_T\cap B_\varepsilon(\hat{p},\hat{t})\} \\
\textmd{and\ } \underline{h}(p,t) & = & \liminf_{\hat{p}\to p,\hat{t}\to t, \varepsilon\to 0} h_\varepsilon(\hat{p}, \hat{t}) \\
 & = & \lim_{\varepsilon\to 0} \inf_{0<\delta<\varepsilon}\{h_\delta(\hat{p},\hat{t}): O_T\cap B_\varepsilon(\hat{p},\hat{t})\} \\
\end{eqnarray*}
\end{definition}
Taking the relaxed limits as $h\to 1^+$ of the operator $F^h_\infty(\nabla_0 u,(D^2u)^{\star})$ in Equation \ref{relax}, we have via the continuity of the operator
\begin{eqnarray*}
\overline{F}^1_\infty(\nabla_0 u,(D^2u)^{\star}) = \underline{F}^1_\infty(\nabla_0 u,(D^2u)^{\star})=
\left\{\begin{array}{cl}
-\|\nabla_0u\|^{-2}\ip{(D^2u)^{\star}\nabla_0u}{\nabla_0u} & \nabla_0u \neq 0 \\
0 & \nabla_0u = 0. \end{array}\right.
\end{eqnarray*}
We give this operator the label $\mathcal{F}(\nabla_0 u,(D^2u)^{\star})$. 
Consider the relaxed limits $\overline{u}(p,t)$ and $\underline{u}(p,t)$ of the sequence of unique (continuous) viscosity solutions to Problem \ref{mainprob} $\{u_h(p,t)\}$ as $h\to 1^+$. By  \cite[Thm 2.2.1]{Giga}, we have $\overline{u}(p,t)$ is a viscosity subsolution and $\underline{u}(p,t)$ is a viscosity supersolution to 
$$u_t+\mathcal{F}(\nabla_0 u,(D^2u)^{\star})=0.$$
We have the following comparison principle, whose proof is similar to Theorem \ref{pinf} in the case to $h=1$ and is omitted. 
\begin{lemma}
Let $\Omega$ be a bounded domain in $G$.   If $\mathfrak{u}$ is a parabolic viscosity subsolution and $\mathfrak{v}$ a parabolic viscosity supersolution to $$u_t+\mathcal{F}(\nabla_0 u,(D^2u)^{\star})=0.$$ then $\mathfrak{u} \leq \mathfrak{v}$ on $\Omega_T\equiv\Omega\times [0,T) $.
\end{lemma}
\begin{corollary}
$\overline{u}(p,t)=\underline{u}(p,t)$.
\end{corollary}
\begin{proof}
By construction, $\underline{u}(p,t)\leq \overline{u}(p,t)$. By the Lemma, $\underline{u}(p,t)\geq \overline{u}(p,t)$.
\end{proof}
Using the corollary, we will call this common relaxed limit $u^1(p,t)$. By \cite[Chapter 2]{Giga} and \cite[Section 6]{CIL:UGTVS}, it is continuous and the sequence $\{u_h(p,t)\}$ converges locally uniformly to $u^1(p,t)$ as $h\to 1^+$. 

We then have the following theorem.
\begin{thm}\label{oneexist}
There exists a unique parabolic viscosity infinite solution to Problem \ref{mainprob} when $h=1$. 
\end{thm}
\begin{proof}
Let $\{u_h(p,t)\}$ and $u^1(p,t)$ be as above. Let $\{h_j\}$ be a subsequence with $h_j\to 1^+$ where 
$u_h(p,t)\to u^1(p,t)$ uniformly. We may assume $h_j<3$. 

Let  $\phi\in\mathcal{A}u_1(p_0,t_0)$. Using the uniform convergence, there is a sequence $\{p_j,t_j\}\to (p_0,t_0)$ so that $\phi \in \mathcal{A}u_{h_j}(p_j,t_j)$. If $\nabla_0\phi(p_0,t_0)\neq 0$, we have  $\nabla_0\phi(p_j,t_j)\neq 0$ for sufficiently large $j$. 
We then have
$$\phi_t(p_j,t_j)-\Delta^{h_j}_\infty \phi(p_j,t_j) \leq 0$$ and letting $j\to\infty$ yields
$$\phi_t(p_0,t_0)-\Delta^1_\infty \phi(p_0,t_0) \leq 0.$$  
Suppose $\nabla_0\phi(p_0,t_0)=0$. By Corollary \ref{moddef}, we may assume $(D^2\phi)^\star(p_0,t_0)=0$.  Suppose passing to a subsequence if needed, we have  $\nabla_0\phi(p_j,t_j)\neq 0$.
Then 
\begin{equation*}
\phi_t(p_j,t_j)-\D \max_{\|\eta\|=1}\ip{(D^2\phi)^\star(p_j,t_j)\;\eta}{\eta}  \leq 
\phi_t(p_j,t_j)-\Delta^{h_j}_\infty \phi(p_j,t_j) \leq 0.
\end{equation*}
Letting $j\to\infty$ yields
$$\phi_t(p_0,t_0)=\phi_t(p_j,t_j)-\D \max_{\|\eta\|=1}\ip{(D^2\phi)^\star(p_0,t_0)\;\eta}{\eta}\leq 0.$$
In the case $\nabla_0\phi(p_j,t_j)= 0$, since $h_j<3$, we have $\phi_t(p_j,t_j) \leq 0$ and letting $j\to\infty$ yields $\phi_t(p_0,t_0)\leq 0$. We conclude that $u_1$ is a parabolic viscosity h-infinite subsolution. Similarly, 
$u_1$ is a parabolic viscosity h-infinite supersolution.
\end{proof}

\section{ The limit as $t\to\infty$.}
We now focus our attention on the asymptotic limits of the parabolic viscosity h-infinite solutions. We wish to show that for $1 \leq h$, we have the (unique) viscosity solution to $u_t-\Delta^h_\infty u=0$ approaches the viscosity solution of $-\Delta^h_\infty u=0$ as $t\to\infty$.  Our goal is the following theorem:
\begin{thm}\label{final} Let $h>1$ and $u \in C(\overline{\Omega} \times [0,\infty))$ be a viscosity solution of 
\begin{equation}\label{finaleqn}
\left\{ \begin{array}{cl}
u_t -\Delta^h_\infty u =0 & \hspace{10pt} in \hspace{5pt} \Omega \times (0,\infty),\\ 
u(p,t) =g(p) & \hspace{10pt} on \hspace{5pt} \p_{\pb}(\Omega \times (0, \infty))
\end{array}\right.
\end{equation}
with $g:\overline{\Omega} \to \mathbb{R}$ continuous and assuming that $\partial \Omega$ satisfies the property of positive geometric density (see \cite[pg. 2909]{J:PD}).  Then $u(p,t) \to U(p)$ uniformly in $\Omega$ as $t \to \infty$ where $U(p)$ is the unique viscosity solution of $-\Delta^h_\infty U = 0$ with the Dirichlet boundary condition $lim_{q \to p} U(q) = g(p)$ for all $p \in \partial \Omega$.
\end{thm}

We first must establish the uniqueness of viscosity solutions to the limit equation. Note that for future reference, we include the case $h=1$. 
\begin{thm}
Let $1\leq h< \infty$ and let $\Omega$ be a bounded domain. Let $u$ be a viscosity subsolution to $\Delta^h_\infty u =0$ 
and let $v$ be a viscosity supersolution to $-\Delta^h_\infty u =0$.
Then,
\begin{equation*}
\sup_{p \in \overline{\Omega}} (u(p)-v(p)) = \sup_{p \in \partial \Omega} (u(p)-v(p)).
\end{equation*}
\end{thm} 
\begin{proof}
Let $u$ be a viscosity subsolution to $-\Delta^h_\infty u =0$. Then choose $\phi \in \ccs(\Omega)$ such that  $0 = \phi(p_0) - u(p_0) < \phi(p) - u(p)$ for $p \in \Omega$, $p \neq p_0$. 
If $\|\nabla_0 \phi(p_0)\| = 0$, then $-\ip{(D^2\phi)^\star(p_0)\nabla_0\phi(p_0)}{\nabla_0\phi(p_0)} = 0 \leq 0$.
If $\|\nabla_0 \phi(p_0)\| \neq 0$, we then have
$$-\Delta^h_\infty \phi(p_0) = -\|\nabla_0\phi(p_0)\|^{h-3}\ip{(D^2\phi)^\star(p_0)\nabla_0\phi(p_0)}{\nabla_0\phi(p_0)} \leq 0.$$ 
Dividing, we have $-\ip{(D^2\phi)^\star(p_0)\nabla_0\phi(p_0)}{\nabla_0\phi(p_0)} \leq 0$.  In either case, $u$ is a viscosity subsolution to $-\Delta^3_\infty u=0$. Similarly, $v$ is a viscosity supersolution to $-\Delta^3_\infty u=0$. The theorem follows from the corresponding result for $-\Delta^3_\infty u=0$ in \cite{B:MP, B:HG, W:W}.  
\end{proof}
We state some obvious corollaries:
\begin{corollary}
Let $1\leq h< \infty$ and let $g:\partial \Omega \to \mathbb{R}$ be continuous. Then there exactly  one solution to 
\begin{equation*}
\left\{ \begin{array}{cl}
-\Delta^h_\infty u =0 & \hspace{10pt} in \hspace{5pt} \Omega \\ 
u=g & \hspace{10pt} on \hspace{5pt} \partial \Omega. 
\end{array}\right.
\end{equation*}
\end{corollary}
\begin{corollary}\label{threeisone}
Let $1\leq h< \infty$ and let $g:\partial \Omega \to \mathbb{R}$ be continuous. The unique viscosity solution to 
\begin{equation*}
\left\{ \begin{array}{cl}
-\Delta^h_\infty u =0 & \hspace{10pt} in \hspace{5pt} \Omega \\ 
u=g & \hspace{10pt} on \hspace{5pt} \partial \Omega 
\end{array}\right.
\end{equation*}
is the unique viscosity solution to 
\begin{equation*}
\left\{ \begin{array}{cl}
-\Delta^3_\infty u =0 & \hspace{10pt} in \hspace{5pt} \Omega \\ 
u=g & \hspace{10pt} on \hspace{5pt} \partial \Omega. 
\end{array}\right.
\end{equation*}
\end{corollary}

Our method of proof for Theorem \ref{final} follows that of \cite[Theorem 2]{J:PD}, the core of which hinges on the construction of a parabolic test function from an elliptic one.  In order to construct such a parabolic test function, we need to examine the homogeneity of Equation \eqref{finaleqn}. A quick calculation shows that for a fixed $h>1$, $k^{\frac{1}{h-1}}u(x,kt)$ is a $\ccs$ solution to Equation \eqref{finaleqn} if $u(x,t)$ is a $\ccs$ solution. A routine calculation then shows parabolic viscosity h-infinite solutions share this homogeneity.  We use this property in the following lemma, the proof of which can be found in  \cite[pg. 170]{DiB}. (Also, cf. \cite[Lemma 6.2] {BM} and \cite{J:PD}.)
\begin{lemma}\label{bounduth}
Let $u$ be as in Theorem \ref{final} and $h>1$.  Then for every $(x,t) \in \Omega \times (0,\infty)$ and for $0 < \mathcal{T}< t$, we have
\begin{equation*}
|u(x,t-\mathcal{T}) - u(x,t)|  \leq  \frac{2 ||g||_{\infty,\Omega}}{h-1}\left(1-\frac{\mathcal{T}}{t}\right)^{\frac{h}{1-h}}\frac{\mathcal{T}}{t} 
\end{equation*}
\end{lemma}
\begin{proof} \textbf{[Theorem \ref{final}]} Fix $h>1$. Let $u$ be a viscosity solution of \eqref{finaleqn}.  The results of \cite[Chapter III]{DiB} imply that the family $\{u(\cdot,t): t \in (0,\infty)\}$ is equicontinuous.  Since it is uniformly bounded due to the boundedness of $g$, Arzela-Ascoli's theorem yields that there exists a sequence $t_j \to \infty$ such that $u(\cdot,t_j)$ converge uniformly in $\overline{\Omega}$ to a function $U \in C(\overline{\Omega})$ for which $U(p) = g(p)$ for all $p \in \partial \Omega$.  Since it is known from \cite[Lemma 5.5]{B:MP} that the Dirichlet problem for the subelliptic $\tp$-Laplace equation possesses a unique solution, it is enough to show that $U$ is a viscosity $\tp$-subsolution to  $-\Delta_\tp U = 0$ on $\Omega$.  With that in mind, let $p_0 \in \Omega$ and choose $\phi \in \ccs(\Omega)$ such that  $0 = \phi(p_0) - U(p_0) < \phi(p) - U(p)$ for $p \in \Omega$, $p \neq p_0$.  Using the uniform convergence, we can find a sequence $p_j \to p_0$ such that $u(\cdot,t_j) - \phi$ has a local maximum at $p_j$.  Now define
$$\phi_j(p,t) = \phi(p) + C\left(\frac{t}{t_j}\right)^{\frac{h}{1-h}}\frac{t_j-t}{t_j},$$
where $C = 2||g||_{\infty,\Omega}/(h-1)$.  Note that $\phi_j(p,t) \in \ccs(\Omega \times (0, \infty))$.  Then using Lemma \ref{bounduth},
\begin{eqnarray*}
u(p_j,t_j) - \phi_j(p_j,t_j) &=& u(p_j,t_j) - \phi(p_j) \geq u(p,t_j) - \phi(p)\\
&\geq& u(p,t) - \phi(p) - C\left(\frac{t}{t_j}\right)^{\frac{h}{1-h}}\frac{t_j-t}{t_j}\\
&=& u(p,t) - \phi_j(p,t)
\end{eqnarray*}
for any $p \in \Omega$ and $0<t < t_j$.  Thus we have that $\phi_j$ is an admissible test function at $(p_j,t_j)$ on $\Omega \times [0,T]$.  Therefore, 
$$(\phi_j)_t(p_j,t_j)-\Delta^h_{\infty}\phi_j(p_j,t_j) \leq 0.$$
This yields 
$$- \Delta^h_{\infty}\phi(p_j) \leq \frac{C}{t_j}.$$
The theorem follows by letting $j \to \infty$.
\end{proof}

Combining the results of the previous sections, we have the following theorem:
\begin{thm}
The following diagram commutes: 
$$\begin{CD}
u^{h,t}_t-\Delta^h_{\infty}u^{h,t}=0 @>>{h\to 1^+}>u^{1,t}_t-\Delta^1_{\infty}u^{1,t}=0 \\
@VV{t\to \infty}V                @VV{t\to\infty}V \\
-\Delta^h_{\infty}u^{h,\infty}=0 @>>{h\to 1^+}> -\Delta^1_{\infty}u^{1,\infty}=0
\end{CD}$$
\end{thm}
\begin{proof}
By Theorem \ref{finaleqn}, Corollary \ref{threeisone}, and Theorem \ref{oneexist}, the top, bottom and left limits exist, with the left limit being a uniform limit. By results of iterated limits (see, for example, \cite{Ba}), we have the fourth limit exists, as does the full limit. In particular, 
$$\lim_{\stackrel{h \to 1^+}{t\to\infty}}u^{h,t}=\lim_{h \to 1^+}\lim_{t\to\infty}u^{h,t}=\lim_{t\to\infty}\lim_{h \to 1^+}u^{h,t}=u^{1,\infty}$$

\end{proof}

\end{document}